\documentclass[11pt]{amsart}
\usepackage{amssymb}

\usepackage[all]{xy}
\theoremstyle{plain}
\newtheorem{theorem}{Theorem}[section]
\newtheorem{lemma}[theorem]{Lemma}
\newtheorem{proposition}[theorem]{Proposition} 
\newtheorem{corollary}[theorem]{Corollary}

\theoremstyle{remark}
\newtheorem{remark}[theorem]{Remark}

\theoremstyle{definition}
\newtheorem{definition}[theorem]{Definition} 
\newtheorem{notation}[theorem]{Notation}

\newcommand{\cL}{\mathcal{L}}
\newcommand{\cM}{\mathcal{M}}
\newcommand{\cO}{\mathcal{O}}
\newcommand{\cN}{\mathcal{N}}
\newcommand{\cE}{\mathcal{E}}

\newcommand{\cF}{\mathcal{F}}
\newcommand{\cT}{\mathcal{T}}

\newcommand{\C}{\mathbf{C}}

\newcommand{\Z}{\mathbf{Z}}
\newcommand{\Ps}{\mathbf{P}}

\DeclareMathOperator{\smooth}{smooth}
\DeclareMathOperator{\aut}{Aut}

\DeclareMathOperator{\NS}{NS}

\DeclareMathOperator{\sing}{sing}

\DeclareMathOperator{\coker}{coker}
\DeclareMathOperator{\rank}{rank}

\DeclareMathOperator{\prim}{prim}
\DeclareMathOperator{\Pic}{Pic}

\DeclareMathOperator{\Gr}{Gr}

\DeclareMathOperator{\cl}{cl}
\DeclareMathOperator{\Sym}{Sym}
\DeclareMathOperator{\Ima}{Im}
\numberwithin{equation}{section}

\title[Chevalley-Weil for elliptic surfaces]{Chevalley-Weil formula for hypersurfaces in $\mathbf{P}^n$-bundles over curves and Mordell-Weil ranks in function field towers}

\author[R.~Kloosterman]{Remke Kloosterman}
\address{Institut f\"ur Mathematik, Humboldt-Universit\"at zu Berlin,
Unter den Linden 6, D-10099 Berlin, Germany} 
\email{klooster@math.hu-berlin.de}
\date{\today}
\thanks{The author would like to thank Gavril Farkas for pointing out the existence of \cite{CW}, Mike Roth for various discussions on the best way to calculate the cohomology of $\Omega^i_{\Ps(\cE)}(kX)$ for a hypersurface $X\subset \Ps(\cE)$ and Orsola Tommasi for various remarks on a previous version of this paper.}
\subjclass{14J27}
\begin{document}

\begin{abstract}
Let $X$ be a complex hypersurface  in a $\Ps^n$-bundle over a curve $C$. Let $C'\to C$ be a Galois cover with group $G$. 
In this paper we describe the $\C[G]$-structure of $H^{p,q}(X\times_C C')$ provided that $X\times_C C'$ is either smooth or $n=3$ and $X\times_C C'$ has at most ADE singularities.

As an application we obtain a geometric proof for an upper bound by P\'al for the Mordell-Weil rank of an elliptic surface obtained 
by a Galois base change of another elliptic surface. 
\end{abstract}

\maketitle

\section{Introduction}\label{secIntro}
Let $k$ be a field of characteristic zero, $C/k$ a smooth, geometrically integral curve, and let $f:C'\to C$ be a (ramified) Galois cover with Galois group $G$.  
Let $E/k(C)$ be a non-isotrivial elliptic curve, i.e., with $j(E)\in k(C)\setminus k$ 
and let $\pi:X\to C$ be the associated relatively minimal elliptic surface with section. Let $R\subset C$ be the set of points over which $f$ is ramified and let $s$ be the number of points in $R$. Let $e$ be the Euler characteristic of $C\setminus R$, i.e., $e=2-2g(C)-s$.

Assume that the discriminant of $\pi$ does not vanish at any point in $R$. Let $c_E$ and $d_E$ be the degree of the conductor of $E/k(C)$ and the degree of the minimal discriminant of $E$, respectively.
P\'al showed in \cite{PalMW} using equivariant Grothendieck-Ogg-Shafarevich theory 
that
\begin{equation}\label{mainEqn} \rank E(k(C'))\leq \epsilon(G,k) (c_E-d_E/6-e)\end{equation}
where $\epsilon(G,k)$ is the Ellenberg constant of $(G,k)$, for a definition see \cite{EllMW}. 
This constant depends only on the group $G$ and the field $K$, but not on $E$.
In this paper we will gibe an alternative proof for this bound

This paper grew out of an  attempt to generalize  this simplified approach to the case where $g(C')>0$.

As noted in \cite{PalMW} it suffices to prove that $E(k(C'))$ is a quotient of a free $k[G]$-module of rank $c_E-d_E/6-e$, and by the Lefschetz principle it suffices to prove this slightly stronger statement only in the case $k=\C$.

Let $X'=\widetilde{X\times_C C'}$ be the elliptic surface associated with $E/\C(C')$. 
Our starting point is that the following ingredients would lead to a proof for the fact that $E(\C(C'))$ is a quotient of $\C[G]^{\oplus c_E+d_E/6-e}$.
\begin{enumerate}
 \item $E(\C(C'))\otimes \C$ is a quotient of $H^{1,1}(X',\C)$.
 \item Let $\mu$ be the total Milnor number of $X$. Then the kernel of the natural map $H^{1,1}(X',\C)\to E(\C(C'))\otimes \C$ contains $\C^2\oplus \C[G]^{\mu}$.
\item $H^0(K_{C'})^{\oplus 2}$ is a quotient of $\C[G]^{-e}$.
\item $\mu=d_E-c_E$.                                                                                                                                                               
\item The $\C[G]$-structure of $H^{1,1}(X',\C)$ is $\C[G]^{\oplus \frac{5}{6} d_E} \oplus H^0(K_{C'})^{\oplus 2}$.
\end{enumerate}
The first point is a consequence of the Shioda-Tate formula for the Mordell-Weil rank of an elliptic surface and the Lefschetz $(1,1)$-theorem. The second point follows also from the Shioda-Tate formula, but here we need to use our assumptions on the ramification of $f$.
The third point is straightforward (Lemma~\ref{lemCan}), the fourth point is not difficult (Corollary~\ref{corPal}).
Hence the crucial point is to determine the $\C[G]$-structure of $H^{1,1}(X',\C)$.

If $C'$ is rational  and all singular fibers of $X'$ are irreducible then the $\C[G]$-structure of $H^{1,1}(X')$ can be determined as follows: In this case $X'$ is birational to a quasismooth surface $W'\subset \Ps(2k,3k,1,1)$ of degree $6k$. This surface is called the Weierstrass model of $X'$. The co-kernel of the injective map $H^{1,1}(W')_{\prim}\to H^{1,1}(X')$ is two-dimensional, and $G$ acts trivially on this co-kernel. Steenbrink~\cite{SteQua} presented a method to find an explicit basis for $H^{1,1}(W')_{\prim}$ in terms of the Jacobian ideal of $W'$, extending Griffiths' method for hypersurfaces in $\Ps^n$. A straightforward calculation then yields the $\C[G]$-structure of $H^{1,1}(W')$.

If $C'$ is rational, but $X'$ has reducible fibers then there are two possible ways to generalize this result. The first approach uses a deformation argument to show that $X'$ is the limit for $t=0$ of a family $X'_t$ with of elliptic surfaces admitting a $G$-action, such that all for $t\neq 0$ the elliptic fibration on $X'_t$ has only irreducible fibers. The second approach uses a result of Steenbrink \cite{SteAdj} where he extends his method to describe $H^{p,q}(W')_{\prim}$ to the case where, very roughly,   the sheaves of Du Bois differentials and of  Barlet differentials on $W'$ coincide (which holds for Weierstrass models of elliptic surfaces, the precise condition on $W'$ is formulated in \cite{SteAdj}).

The above approaches can be extended to many cases where $C'$ is not rational. Let $\pi: X\to C$ be an elliptic surface, and let $S\subset X$ be the image of the zero section. Let $N_{S/X}$ be the normal bundle of $S$. Then one can find a Weierstrass model $W$ of $X$ in $\Ps(\cE)$ where $\cE=\cO\oplus \cL^{-2}\oplus \cL^{-3}$, with  $\cL=(\pi_* N_{S/X})^*$. Similarly the Weierstrass model of the base changed elliptic surface is a surface $W'$ in $\Ps(f^*\cE)=:\Ps$.
The Griffiths-Steenbrink approach yields two injective maps
\[ \frac{H^0(K_{\Ps}(2W'))}{H^0(K_{\Ps}(W')\oplus dH^0(\Omega^2(2W'))} \hookrightarrow H^{1,1}(W')\hookrightarrow H^{1,1}(X').\]
Using our assumptions on $f$ we can easily describe the $\C[G]$-action on the left hand side. The cokernel of the second map  is isomorphic to $ \C[G]^\mu$. The dimension of the cokernel of the first map is $2+h^1(f^*\cL)$. The 2 corresponds to two copies of the trivial representation, however, it is not that easy to describe the $\C[G]$-action on the vector space of dimension $h^1(f^*\cL)$. From this it follows that the Griffiths-Steenbrink approach works as long as $h^1(f^*\cL)$ vanishes. This happens only if the degree of the ramification divisor $C'\to C$ is small compared to $\deg(f)$ and $\deg(\cL)$.

To avoid  this restriction on $h^1(\cL)$ we work with equivariant Euler characteristic: Let $K(\C[G])$ be the Grothendieck group of all finitely generated $\C[G]$-modules. For a coherent sheaf $\cF$ on a scheme  with a $G$-action one defines
\[ \chi_G(\cF)=\sum_i (-1)^i [H^i(X,\cF)].\]
We   use the ideas behind the Griffiths-Steenbrink approach to prove that the class of $H^{1,1}(W')$ in $K(\C[G])$ equals
 \[ 2[\C]-\chi_G(\Omega^2_{\Ps}(W'))+\chi_G(K_{\Ps}(2W'))-\chi_G(H^0(\cT))-\chi_G(K_{\Ps}(W')).\]
Here $\cT$ is a skyscraper sheaf supported on the singular locus of $W'$, such that its stalk is isomorphic to the Tjurina algebra of the singularity, and $\Omega^{2,\cl}_{\Ps}$ is the sheaf of closed 2-forms.
The remaining Euler characteristics can be calculated by fairly standard techniques and thereby yieling a proof of the point (5) mentiond above.

One can easily describe $H^{1,1}(X')$ (as $\C[G]$-module) in terms of the regular representation $\C[G]$ and $H^{1,1}(W')$. The $\C[G]$-structure on the other $H^{p,q}(X')$ can be determined by standard techniques. In the sequel we show:  
\begin{proposition}\label{prpCGstr} Let $\pi: X\to C$ be an elliptic surface and set $\cL=(\pi_*N_{S/X})^*$.
Let $f:C'\to C$ be a ramified Galois cover with group $G$ and let $X'\to C'$ be the smooth minimal elliptic surface birational to $X\times_C C'$. Suppose that over each branch point of $f$ the fiber of $\pi$ is smooth or semistable. Then we have the following identities in $K(\C[G])$:
\begin{eqnarray*}
[ H^{0,1}(X',\C)]&=&[H^{1,0}(X',\C)]=[ H^0(C',K_{C'})]; \\ 
{[H^{2,0}(X',\C)]}&=&[H^0(C',K_{C'})]-[\C]+\deg(\cL) [\C[G]] \\ 
{[H^{1,1}(X',\C)]}&=&2[H^0(C',K_{C'})]+10\deg(\cL)[\C[G]]
\end{eqnarray*}
 \end{proposition}
Since $X'$ is smooth we can  use Poincar\'e duality to describe the $\C[G]$-structure of $H^{p,q}(X')$ for all other $p,q$.
As argued above, this Proposition is sufficient to prove the bound (\ref{mainEqn}), see Corollary~\ref{corPal}.

Our approach to determine the $\C[G]$-structure of $H^{p,q}$ works for a larger class of varieties:
\begin{theorem}\label{mainThm}
Let $C$ be a smooth projective curve and $\cE$ a rank $r$ vector bundle over $C$, which is a direct sum of line bundles. Let $X\subset \Ps(\cE)$ be a hypersurface. Let $f:C'\to C$ be a Galois cover and let $X'=X\times_C C'$. Assume that either $X'$ is  smooth or $r=3$ and $X'$ is a surface with at most ADE singularities.

Moreover, assume  $H^i(X')\cong H^i(\Ps(f^*\cE))$ for $i\leq r-2$.

Then we have the following identity in $K(\C[G])$ 
\[ [H^{p,q}(X')]=a [\C[G]]+b\chi_G(\cO_C)+c[\C]+d [H^0(\cT)]\]
for some integers $a,b,c,d$, which can be determined explicitly and depend on $p$, $q$, the degrees of the direct summands of $\cE$ and the fiber degree of $X$.
\end{theorem}

We would like to make one remark concerning the bound of P\'al: If each of the elements of $G$ is defined over $k$, then the Ellenberg constant equals the number of elements of $G$. In this case one easily shows that the above bound is weaker than the bound obtained by the Shioda-Tate formula. However, if the Galois group of $k$ acts highly non-trivially on $G$ then the Ellenberg constant is small and therefore this bound is very useful.

There are many other results on the behaviour of the Mordell-Weil rank under base change. Most of these results assume that the fibers over the critical values are very singular, e.g., the results by Fastenberg \cite{Fab,Fa,Fac} and by Heijne \cite{Heijne}. 
Bounds in the case where the fibers over the critical values are smooth and where the base change map is \'etale, are obtained by Silverman \cite{SilvTower}. Ellenberg proved a slightly weaker bound in a much more general setting, namely he showed that
\[ \rank E(k(C'))\leq \epsilon(G,\Sigma) (c_E-2e)\]
without imposing any condition on $G$, and assuming only that 6 is invertible in $k$.

The $\C[G]$-structure of the cohomology of a ramified cover $X\to Y$ has been studied in general, but we could not find any result that was sufficiently precise to prove (\ref{mainEqn}). The first result in this direction was by Chevalley-Weil \cite{CW} in the curve case. There are  several results by Nakajima in the higher-dimensional case \cite{Naka}.

In Section~\ref{sectWeiMod} we discuss the construction of Weierstrass models associated with elliptic surfaces. In Section~\ref{sectHpq} we prove Theorem~\ref{mainThm}. In Section~\ref{sectWeiModCG} we determine the constants $a,b,c,d$ for the case of Weierstrass models of elliptic surfaces and give a proof for (\ref{mainEqn}).

\section{Weierstrass models and Projective bundles}\label{sectWeiMod}
In this section let $C$ be a smooth projective curve and  $\cL$  a line bundle on a curve $C$, of positive degree.
Let $\cE=\cO\oplus \cL^{-2}\oplus\cL^{-3}$, let $\Ps(\cE)$ be the associated projective bundle, parametrizing one-dimensional quotients of $\cE$. Let $\varphi: \Ps\to C'$ be the projection map. Then $\varphi_*(\cO_{\Ps}(1))=\cE$. 
Let \begin{eqnarray*}X&=&(0,1,0)\in H^0( \varphi^*\cL^2(1))=H^0(\cL^2)\oplus H^0(\cO_C)\oplus H^0(\cL^{-1})\\Y&=&(0,0,1)\in H^0(\varphi^*\cL^3(1))=H^0(\cL^3)\oplus H^0(\cL)\oplus H^0(\cO_C)\\Z&=&(1,0,0)\in H^0(\cO_{\Ps}(1))=H^0(\cO\oplus \cL^{-2}\oplus \cL^{-3})\end{eqnarray*} be the standard coordinates.

\begin{definition}
A (minimal) Weierstrass model $W$ is an element
\[ F:=-Y^2Z-a_1XYZ-a_3YZ^2+X^3+a_2X^2Z+a_4XZ^2+a_6Z^3\]
in $|\cL^6\otimes \cO_{\Ps(\cE)}(3)|$, such that $V(F)\subset \Ps(\cE)$ has at most isolated ADE singularities. 
\end{definition}
\begin{notation}
The restriction of $\varphi$ to a Weierstrass model $W$ is a morphism with only irreducible fibers, and the generic fiber is a genus one curve.
For a fixed Weierstrass model $W$ denote with $X$ its minimal resolution of singularities and with $\pi:X\to C$ the induced fibration.
\end{notation}

\begin{lemma}\label{lemWM} The minimal resolution of singularities of a Weierstrass model is an elliptic surface $\pi: X\to C$. The section $\sigma_0:C\to W$, which maps a point $p$ to the point $[0:1:0]$ in the fiber over $p$, extends to a section $C\to X$.
\end{lemma}
\begin{proof}
 The first statement is straightforward. From the shape of the polynomial $F$ it follows that $W_{\sing}$ is contained in $V(Y)$.
 Recall that $\sigma_0(C)=V(X,Z)$. Hence $\sigma_0(C)$ does not intersect $W_{\sing}$ and we can extend $\sigma_0:C\to X$.
\end{proof}
\begin{remark} 
Conversely, every elliptic surface over $C$ admits a minimal Weierstrass model for a proper choice of line bundle $\cL$, namely $\cL$ is the inverse of the push forward of the normal bundle of the zero section. The line bundle $\cL$  is of non-negative degree. If the degree of $\cL$ is zero then the fibration has no singular fibers and after a finite \'etale base change the elliptic surface is a product. See \cite[Section III.3]{MiES}.
\end{remark}

\begin{remark}\label{rmkShort}
Since we work in characteristic zero we may, after applying an automorphism of $\Ps(\cE)/C$ if necessary, assume that $a_1,a_2$ and $a_3$  vanish. In the sequel  we work with a short Weierstrass equation \[-Y^2Z+X^3+A XZ^2+B Z^3\] with $A\in H^0(\cL^4)$ and $B\in H^0(\cL^6)$.

This is the equation of a minimal Weierstrass model if and only if for each point $p\in C$ we have either $v_p(A)\leq 3$ or $v_p(B)\leq 5$.
\end{remark}

\begin{lemma}\label{lemWMsmooth} The Weierstrass model $W$ is smooth if and only if all singular fibers of $\pi$ are of type $I_1$ and $II$.
\end{lemma}
\begin{proof}
The Weierstrass model $W$ is smooth if and only if $X\cong W$. Since all fibers of $W\to C$ are irreducible, this is equivalent to the fact that all singular fibers of $\pi$ are irreducible. Hence these fibers are of type $I_1$ or $II$.\end{proof}

\begin{lemma}\label{lemWMbasechange} Let $W$ be a Weierstrass model with associated line bundle $\cL$. Let $f:C'\to C$ be a finite morphism of curves.
Suppose that over the branch points of $f$ the fiber of $\pi$ is either smooth or semi-stable.

Then $W':=W\times_C C'$ is a Weierstrass model (with associated line bundle $f^*(\cL)$).
\end{lemma}
\begin{proof}
Consider the induced map $\Ps(f^*(\cE))\to \Ps$. Then $W'$ is the zero set of
\[-Y^2Z+X^3+f^*(A)XZ^2+f^*(B)Z^3.\]
If $W'$ is not a Weierstrass model then there is a point $p\in C'$ such that $v_p(f^*(A))\geq 4$ and $v_p(f^*(B))\geq 6$. 

Since $W$ is Weierstrass model we have  $v_q(A)\leq 3$ or $v_q(B)\leq 5$ for all $q\in C$. 
Let $e_p$ be the ramification index of $p$ then  $v_p(f^*A)=e_p v_q(A)$ and $v_p(f^*B)=e_p v_p(B)$ for $q=f(p)$. 
Hence if $v_p(f^*A)\geq 4$ and $v_p(f^* B)\geq 6$ then $e_p>1$, i.e. $f$ is ramified at $p$. However, in this case the fiber of $f(p)$ is either smooth or  multiplicative. This implies that at least one of $A(q)$ or $B(q)$ is nonzero. Hence at least one $v_p(f^*A)$ or $v_p(f^*B)$ vanishes and therefore $W'$ is a minimal Weierstrass model. 
\end{proof}

Since $W$ has only ADE singularities we have that the cohomology of $W$ and $X$ are closely related:
\begin{proposition}\label{prpDesing} Let $W$ be a Weierstrass model and $\pi:X\to C$ the elliptic fibration on the minimal resolution of singularities of $W$.
Let $\mu$ be the total number of fiber-components of $\pi$ which do not intersect the image of the zero-section. Then $\mu$ equals the total Milnor number of the singularities of $X$.

Moreover,  the natural mixed Hodge structure on $H^i(W)$ is pure for all $i$ and we have $h^{p,q}(X)=h^{p,q}(W)$ for $(p,q)\neq (1,1)$ and $h^{1,1}(X)=h^{1,1}(W)+\mu$.
\end{proposition}

\begin{proof}
 All fibers of $W \to C$ are irreducible by construction. Hence the number of fiber components not intersecting the image of the zerosection equals the number of irreducible components of the exceptional divisor $X\to W$. 
 
 The resolutions of ADE surfaces singularities are well-known, and the number of irreducible components of the exceptional divisor equals the Milnor number, proving the first claim.

 The intersection graph of the exceptional divisor of a resolution of an ADE singularity is also well-known and from this it follows that the exceptional divisors are simply connected complex curves. Hence if  we have $s$ singular points with total Milnor number $\mu$ and $E$ is the total exceptional divisor then  $H^0(E)=\C^s$ and $H^2(E)=\C(-1)^\mu$ and all other cohomology groups vanish.
  
  Let $\Sigma=W_{\sing}$. From \cite[Corollary-Definition 5.37]{PSbook} it follows that we have a long exact sequence of MHS
 \begin{equation}\label{PSes} \dots \to H^i(W)\to H^i(X)\oplus H^i(\Sigma) \to H^i(E) \to H^{i+1}(W)\to \dots\end{equation}
Note that $h^i(\Sigma)=0$ for $i \neq 0$. Moreover, the map $H^0(\Sigma)\to H^0(E)$ is clearly an isomorphism, combining this with the fact that $H^i(E)=0$ for $i\neq 0, 2$ we obtain that $H^i(X)\cong H^i(W)$ for $i\neq 2,3$.

 To prove the proposition it suffices to show that the map $H^2(E)\to H^3(W)$ is zero. As $H^2(E)=\C(-1)^\mu$ has a pure Hodge structure of weight 2 it suffices to show that all the nontrivally Hodge weights of $H^3(W)$ are at least 3. If $W$ is smooth then this is trivially true, so suppose that $W$ is singular.
 
 Consider  the long exact sequence of the pair $(W,W_{\smooth})$. Since $W$ has only ADE singularities and the dimension of $W$ is even it follows that $H^i_\Sigma(W)=0$ for $i\neq 4$, and $H^4_{\Sigma}(W)=\C(-2)^s$. The long exact sequence of the pair $(W',W'_{\smooth})$ now yields isomorphisms 
 $H^i(W)\cong H^i(W_{\smooth})$ for $i\neq 3,4$ and  an exact sequence
 \[ 0\to H^3(W)\to H^3(W_{\smooth}) \to \C(-2)^{\#\Sigma} \to H^4(W')\to 0=H^4(W_{\smooth})\]
 Since $W_{\smooth}$ is smooth we have that the Hodge weights of $H^3(W_{\smooth})$ are at least 3, and hence the same statement holds true for $H^3(W)$. \end{proof}
 
 \begin{lemma}\label{lemLefschetz}
  Consider the inclusion $i:W\to \Ps$. Then $i^*: H^k(\Ps)\to H^k(W)$ is an isomorphism for $k=0, 1,3$, is injective for $k=2$ and is surjective for $k=4$. \end{lemma}
 \begin{proof} 
  For $k=0$ the statement is trivial.
  The case $k=1$ can be shown as follows:
  Consider $\sigma_0: C\to W$ and $i\circ \sigma_0:C\to \Ps$. Combining these morphisms with $\pi: W\to C$, respectively $\psi:\Ps\to C$, yield the identity on $C$.
This implies that $\pi^*\circ \sigma_0^*$ and $\psi^*\circ (i \circ \sigma_0)^*$ are isomorphisms and that $\sigma_0^*:H^k(C)\to H^k(W)$ 
is injective. 

From \cite[Lemma IV.1.1]{MiES} it follows that  $h^1(C)=h^1(X)$  and by the previous proposition we have $h^1(W)=h^1(X)$. In particular $\sigma_0^*$ and $(i\sigma_0)^*$ are  isomorphisms and therefore $i^*$ is an isomorphism.  

For $k=2$ note that $H^2(\Ps)$ is generated by the first Chern classes of a fiber of $\varphi$ and $\cO_{\Ps}(1)$. Their images in $H^2(X)$ are clearly independent, hence the composition $H^2(\Ps)\to H^2(W) \to H^2(X)$ is injective.
For $k=4$ note that the selfintersection of $c_1(\cO_{\Ps}(1))\in H^4(\Ps)$ is mapped to a nonzero element in the one-dimensional vector space $H^4(X)$. Hence $H^4(\Ps)\to H^4(W)\to H^4(X)$ is surjective. Since $H^4(W)\cong H^4(X)$ this case follows also.

The case $k=3$ is slightly more complicated. By successively blowing up points in $\Ps$ we find a variety $\tilde{\Ps}$ such that the strict transform of $W$ is isomorphic with $X$. Now let $H$ be an ample class of $\tilde{\Ps}$ and $H_X$ its restriction to $X$. From the hard Lefschtez theorem it follows that the cupproduct with the class of $H|_X$ induces an isomorphism $H^1(X)\to H^3(X)$. Since $i^*: H^1(\Ps)\to H^1(W)$ is an isomorphism it follows  that $H^1(\tilde{\Ps})\to H^1(X)$ is an isomorphism. Therefore we find a morphsim $H^1(\tilde{\Ps})\to H^3(X)$. We can factor this morphism also as first taking the cupproduct with $H$, and then applying $i$. Hence $i^*:H^3(\tilde{\Ps})\to H^3(X)$ is surjective. Since we blow up only smooth points in $\Ps$ we find $H^3(\tilde{\Ps})=H^3(\Ps)$ and we showed before that $H^3(X)=H^3(W)$. Hence $H^3(\Ps)\to H^3(X)$ is surjective, and is an isomrphism because both vector spaces are of the same dimension.
\end{proof}

\section{The $\C[G]$-structure of $H^{p,q}(X')$}\label{sectHpq}
Let $\cE$ be a rank $n+1$ vector bundle on a smooth curve $C$. Let $X\subset \Ps(\cE)$ be a hypersurface such that either $X$ is smooth or $X$ is a surface with ADE singularities.

 Let $f:C'\to C$ be a Galois cover with group $G$, such that $X':=X\times_C C$ is smooth or $X'$ is a surface with ADE singularities.
 
We now want to describe the $\C[G]$-module structure of $H^{p,q}(X')$. For this we prove first four easy lemmas concerning  identities between representations.
\begin{definition}
For a scheme $Z$ with a $G$-action and a sheaf $\cF$, denote with $\chi_G(\cF)$ the \emph{equivariant Euler characteristic}
\[ 
\sum_i (-1)^i [H^i(Z,\cF)]
\]
in $K(\C[G])$, the Grothendieck group of all finitely generated $\C[G]$-modules.
\end{definition}

In the sequel we use the following lemma, which can be proven by ``the usual devissage argument'' and Serre duality:
\begin{lemma}[{\cite[Lemma 5.6]{PalManCst}}]\label{lemPal} Let $f:C'\to C$ be a ramified Galois cover with group $G$. If $\cM$ is a line bundle on $C$, then
 \[ \chi_G(f^*\cM)=\deg(\cM)\C[G]+\chi_G(\cO_{C'}).\]
 and
 \[ \chi_G(f^*\cM\otimes K_{C'})=\deg(\cM)\C[G]-\chi_G(\cO_{C'}).\]
\end{lemma}

Let $R$ be the set over which $f$ is ramified. If $R$ is non-empty then let $Z$ be the zero-dimensional scheme on $C'$ such that
\begin{equation}\label{eqnram}
 0\to K_{C'}\to f^*K_C(R) \to \cO_Z \to 0
\end{equation}
is exact. Let $s$ be the number of points in $R$.

\begin{lemma}\label{lemCan} Let $f:C'\to C$ be a Galois cover of curves, with group $G$.
If $f$ is unramified then
\[ [H^0(K_{C'})]=[H^0(f^*K_{C})]=[\C]+(g(C)-1)[\C[G]].\]
If $f$ is ramified then
\[ 2[H^0(K_{C'})]+[H^0(\cO_Z)]=2[\C]+(2g(C)-2+s)[\C[G]].\]
\end{lemma}
\begin{proof} If $f$ is ramified then the degree of $f^*K_C(S)$ is strictly larger than $2g(C')-2$, hence its first cohomology group vanishes and we obtain from Lemma~\ref{lemPal} that
\[[ H^0(f^*K_C(S))]=[\C]-[H^0(K_{C'})]+(2g(C)-2+s)[\C[G]].\]
From the exact sequence (\ref{eqnram}) we obtain that
\[ [H^0(K_{C'})]-\C=[H^0(K_{C'})]-[H^1(K_{C'})]=[H^0(f^*K_{C}(S))]-[H^0(\cO_Z)].\]
Combining this yields
\[ 2[H^0(K_{C'})]+[H^0(\cO_Z)]=2[\C]+(2g(C)-2+s)[\C[G]].\]

If $f$ is unramified then $f^*K_C=K_C'$. Lemma~\ref{lemPal} implies now
\[ \chi_G(K_{C'})=\deg(K_C)[\C[G]]+\chi_G(\cO_{C'}).\]
From  $\chi_G(\cO_{C'})=-\chi_G(K_{C'})$ we obtain 
\[2\chi_G(K_{C'})=( 2g(C)-2)[\C[G]].\]
The result now follows from $\chi_G(K_{C'})=[H^0(K_{C'})]-[\C]$.
\end{proof}

\begin{lemma}\label{lemCanS} Let $f:C'\to C$ be a Galois cover of curves, with group $G$.
Then $H^0(K_{C'})^{\oplus 2}$ is a quotient of $\C^{\oplus 2}\oplus \C[G]^{\oplus 2g(C)-2+s}$.
\end{lemma}
\begin{proof}
 This follows directly from the previous lemma.
\end{proof}

\begin{remark}
 The Chevalley-Weil formula gives a precise description of the $\C[G]$-structure of $H^0(K_{C'})$, see \cite{CW}.
\end{remark}

We will now go back to our hypersurface $X'\subset \Ps(f^*(\cE))$. Denote with $\varphi:\Ps(f^*\cE)\to C'$ and $\varphi_0:\Ps(\cE)\to C$ the natural projection maps.

We will now prove a structure theorem for the $\C[G]$-module $H^{p,q}(X')$.
\begin{proposition}\label{prpCG}
Suppose that $\cE$ is a direct sum of line bundles. Let $X\subset \Ps(\cE)$ be a hypersurface, and $X'=X\times_C C'$. Then for $i>0, k\geq 0$ we have that $\chi_G(\Omega^i_{\Ps(f^*\cE)}(kX'))$ is a direct sum of copies of $\C[G]$ and $\chi_G(\cO_{C'})$.
\end{proposition}
\begin{proof}
Let $\varphi:\Ps(f^*(\cE))\to C'$ be the natural projection map.
Consider the short exact sequence
\[ 0\to \varphi^*K_{C'}\to \Omega^1_{\Ps(f^*\cE)} \to \Omega^1_{\varphi}\to 0.\]
On $\Omega^t_{\Ps(f^*\cE)}$ there is a filtration such that $\Gr^p =\wedge^p \varphi^*(K_{C'}) \otimes \Omega^{t-p}_{\varphi}$ \cite[Exer. II.5.16]{Har}. From $\wedge^p \varphi^*K_{C'}=0$ for $p>1$ it follows that at most two of the $\Gr^p$s are nonzero and they fit in the exact sequence
\begin{equation}\label{excOmt} 0 \to \varphi^*(K_{C'})\otimes \Omega^{t-1}_{\varphi}    \to \Omega^t_{\Ps(f^*(\cE))} \to \Omega^t_{\varphi} \to 0 .\end{equation}

Similarly, consider the Euler sequence
\[ 0\to \Omega^1_{\varphi} \to (\varphi^*f^*\cE)(-1) \to \cO_{\Ps(f^*\cE)}\to 0.\]
By using  the filtration constructed in  \cite[Exer. II.5.16]{Har} again we obtain the following exact sequence
\begin{equation}\label{excOm1} 0\to \Omega^t_{\varphi} \to \wedge^t (\varphi^*f^*\cE)(-1) \to \Omega^{t-1}_{\varphi} \to 0.\end{equation}

Let $\cL\in \Pic(C)$ and $d>0$ be such that $\cO_{\Ps(f^*(\cE))}(kX')=(\varphi^*f^*(\cL))(d)$.
A straightforward exercise using the exact sequence (\ref{excOmt}) tensored with $\cO(kX')$, the exact sequence (\ref{excOm1}) tensored with $\cO(kX')$ respectively with $\cO(kX')\otimes \varphi^*(K_{C'})$ and induction on $t$ yields that $\chi_G(\Omega^i_{\Ps(f^*\cE)}(\varphi^*f^*\cL)(d))$ equals
\[ \sum_{i=0}^t (-1)^{t-i}\chi_G((\Lambda_i \otimes \varphi^*f^*\cL)(d))+\sum_{i=0}^{t-1}(-1)^{t-i} \chi_G((\Lambda_i \otimes \varphi^*(f^*\cL\otimes K_{C'}))(d)).\]
with 
\[\Lambda_t:=\wedge^t (\varphi^*f^*\cE)(-1)\]

Using that $R^i\varphi_*(\cO(k))=0$ for $i>0,k\geq -1$ (see \cite{VerPoi}) and the projection formula again we obtain that $\chi_G(\cF)=\chi_G(\varphi_*\cF)$ where $\cF$ is one of
\begin{equation}\label{sheafeq}
 (\wedge^t (\varphi^*f^*\cE)(d-1)) \otimes \varphi^*(f^*(\cL)), (\wedge^t (\varphi^*f^*\cE)(d-1)) \otimes \varphi^*(K_{C'}\otimes f^*(\cL))).
\end{equation}
Since $\cE$ is a sum of line bundles, we obtain that 
\[ (\wedge^t  f^*\cE)\]
is a direct sum of line bundles pulled back from $C$. Similarly we obtain that
\[ R^i \varphi_* \cO(k)=\Sym^k(f^*\cE)\]
is a direct sum of line bundles pulled back from $C$ and by using the projection formula we have that $\varphi_*\cF$ is the direct sum of line bundles pulled back from $C$, for $\cF$ as in (\ref{sheafeq}).

We can therefore calculate the relevant equivariant Euler characteristic by Lemma~\ref{lemPal}, and we obtain that $\chi_G(\varphi_*(\cF))$ is a sum of copies $\chi_G(K_{C'})$ and $\C[G]$ for $\cF$ as in (\ref{sheafeq}). The multiplicity of $\C[G]$ depends on the sum of degrees of the direct summands and the multiplicity of $\chi_G(K_{C'})$ on the rank of $\cF$. Hence the multiplicity of $\chi_G(K_{C'})$ and $\C[G]$ in $\chi_G(\Omega^i(kX'))$ depend only on $i$, $k$, the fiberdegree of $X'$ and the degrees of the direct summand of $\cE$.
\end{proof}

\begin{remark}
 Note that the proof of the theorem also yields a method to determine the number of copies of $\C[G]$, respectively, $\chi_G(\cO)$ which occur. In the next section we make this precise for the case $\cE=\cO\oplus \cL^{-2} \oplus \cL^{-3}$, $X\in |(\varphi^*f^*\cL^6)(3)|$ and $(i,k)=(2,1),(3,1),(3,2)$.
\end{remark}

\begin{proposition}\label{prpHS} Let $n\geq 2$. Let $X\subset \Ps$ be a $n$-dimensional smooth hypersurface. Assume that for $i: X\subset \Ps$ we have that $i^*:H^k(\Ps,\C) \to H^k(X,\C)$ is an isomorphism for $k<n$ and that for $k=n$ this map is injective. Let $U=\Ps\setminus X$. Then $H^i(U)=0$ for $i\neq 0,1,2,n+1$
Moreover, we have isomorphisms $H^0(U)\cong \C$, $H^1(U)\cong H^1(C)$, $H^2(X)\cong \C(-1)$ and $H^n(U)(1)\cong \coker H^{n-1}(\Ps)\to H^{n-1}(X)$ 
\end{proposition}
\begin{proof}
Consider the Gysin exact sequence for cohomology with compact support 
\[\dots \to H^k_c(U)\to H^k_c(\Ps)\to H^k_c(X) \to H^{k+1}_c(U)\to \dots\]
Our assumption on $i^*$ now yields $H^k_c(U)=0$ for $k\leq n$.

Let $\cM$ be an ample line bundle on $\Ps$, and $\cM'$ be its restriction to $X$. Then by the hard Lefschetz theorem we get that the $k$-fold cupproduct with $c_1(\cM')$ yields an isomorphism $H^k(X,\C)\to H^{n-k}(X,\C)$. For $0<k\leq n$ we obtain an isomorphism
\[ H^k(\Ps,C)\to H^k(X,\C) \to H^{n-k}(X,\C)\]
We can factor this isomorphism as first taking the $k$-fold cupproduct with $c_1(\cM)$ and then applying $i^*$. In particular the map $H^{n-k}(\Ps)\to H^{n-k}(X)$ is surjective.
The Betti numbers of $\Ps$ are well-known, namely $h^0(\Ps)=h^{2n+2}(\Ps)=1$, $h^{2k}(\Ps)=2$ for $k=1,\dots, n$ and $h^{2k+1}=h^1(C)$ for $k=0,\dots, n$.
These facts yield that $H^i(\Ps)\cong H^i(X)$ for $i=0,\dots,n-1$ and $i=n+1,\dots 2n-1$.
Hence $H^i_c(U)=0$ for $i\neq n+1, 2n,2n+1,2n+2$. Moreover we have two exact sequences
\[ 0 \to H^n(\Ps)\to H^n(X)\to H^{n+1}_c(U)\to 0\]
and
\[ 0 \to H^{2n}_c(U)\to H^{2n}(\Ps)\to H^{2n}(X)\to 0 \]
and isomoprhisms $H^i_c(U)\cong H^i_c(\Ps)$ for $i=2n+1,2n+2$.

Applying Poincar\'e duality now gives the result.
\end{proof}

Denote with $\Omega^{p,\cl}_{\Ps}$ or $\Omega^{p,\cl}$ the sheaf of closed $p$-forms on $\Ps$.
Recall 
that for a hypersurface $X\subset\Ps$ we have $\Omega^{p,\cl}(X)=\Omega^{p,\cl}(\log X)$.

\begin{proposition}\label{prpHpqChar} 
Let $X\subset \Ps$ be a $n$-dimensional smooth hypersurface. Suppose $n\geq 2$.
Let $G\subset \aut(\Ps,X)$ be a subgroup. Assume that for $i: X\subset \Ps$ we have that $i^*:H^k(\Ps,\C) \to H^k(X,\C)$ is an isomorphism for $k<n$ and that for $k=n$ this map is injective.

Then  for $p\geq 1$ we have 
 $(-1)^{n-p}([H^{p,n-p}(X)]-[ H^{p,n-p}(\Ps)])$ equals
\[ \sum_{k=1}^{n-p+1} (-1)^{k} \chi_G(\Omega^{p+k}(kX))+\sum_{k=1}^{n-p} (-1)^{k} \chi_G(\Omega^{p+1+k}(kX)).\]
and for $p=0$ we find that
\[[H^0(K_{C'})]-[\C]+(-1)^n [H^{0,n}(X)]\] equals 
\[\sum_{k=1}^{n+1} (-1)^k \chi_G(\Omega^{k}(kX))+\sum_{k=1}^{n} (-1)^k \chi_G(\Omega^{k+1}(kX))).\]
\end{proposition}

\begin{proof} Let $U$ be the complement of $X$ in $\Ps$. From the previous proposition it follows that 
\[ [H^{p,n-p}(X)]-[ H^{p,n-p}(\Ps)]=[\Gr_F^{p+1} H^{n+1}(U)].\]
Hence we will focus on determining the $\C[G]$ structure of $\Gr_F^{p+1} H^{n+1}(U)$.

From Deligne's construction of the Hodge filtration on the cohomology of $U$ we get
\[ F^p H^k(U,\C)= \Ima(\mathbf{H}^k(\Omega^{\geq p}_{\Ps(\cE)}(\log X)) \to \mathbf{H}^k(\Omega^\bullet_{\Ps(\cE)}(\log X))).\] 
The map is injective by the degeneracy of the Fr\"ohlicher spectral sequence at $E_1$.
Recall that $\Omega^{p,\cl}(X)$ is the kernel of $d:\Omega^p(X) \to \Omega^{p+1}(2X)$. For $p\geq 1$ we have that the filtered de Rham complex is a resolution of
$\Omega^{p,\cl}(X)$. Combining these fact we obtain for $p\geq 1$ that 
\[ F^p H^{p+q}(U,\C)= H^q(X,\Omega^{p,\cl}(X)).\]

For $p>1$ we have $\Gr_F^p H^{p+q}(U,\C)=0$ except possibly for $q=n+1-p$. In particular, $H^q(\Omega^{p,\cl}(X))=0$ for $q\neq n+1-p$, $p\geq 2$. Hence for $p\geq 2$ we obtain that $\chi_G( \Omega^{p,\cl}(X))$ equals
\[ (-1)^{n+1-p} [H^{n+1-p}(X,\Omega^{p,c}(X))]=(-1)^{n+1-p} F^p H^{n+1}(U,\C).\]

The exact sequence
\[ 0 \to \Omega^{p,\cl}(tX)\to \Omega^{p}(tX)\to \Omega^{p+1,\cl}((t+1)(X))\to 0\]
then yields
\[ \chi_G(\Omega^{p,\cl}(tX))=\sum_{k=0}^{n+1-p} (-1)^k \chi_G(\Omega^{p+k}((t+k)X)).\]
From this we obtain that for $p\geq 1$ we have that $\Gr_F^{p} \coker(H^n(\Ps)\to H^n(X))=\Gr_F^{p+1}H^{n+1}(U)$  equals $(-1)^{n-p}$ times
\[\sum_{k=1}^{n-p+1} (-1)^k \chi_G(\Omega^{p+k}(kX))+\sum_{k=1}^{n-p} (-1)^k \chi_G(\Omega^{p+1+k}(kX)))\]

For $p=0$ we find
\begin{eqnarray*} \chi_G(\Omega^{1,\cl}(X))&=& [F^1 H^1(U,\C)]-[F^1 H^2(U,\C)]+(-1)^{n} [F^1 H^{n+1}(U,\C) ]\\& =& [H^0(\Omega^{1,\cl}(X))]-[H^1(\Omega^{1,\cl})]+(-1)^n[H^{n}(\Omega^{1,\cl}(X))]\end{eqnarray*}
From Proposition~\ref{prpHS} it follows that \[ [F^1 H^1(U,\C)]=[H^0(K_{C'}))]\mbox{ and }[F^1 H^2(U,\C)]=[\C]\] holds. As above we find that \[[H^0(K_{C'})]-[\C]+(-1)^n [\Gr_F^{0} \coker(H^n(\Ps)\to H^n(X)]\] equals 
\[\sum_{k=1}^{n+1} (-1)^k \chi_G(\Omega^{k}(kX))+\sum_{k=1}^{n} (-1)^k \chi_G(\Omega^{k+1}(kX))).\]
 \end{proof}

Let $P$ be smooth compact K\"ahler manifold. 
 Steenbrink \cite{SteAdj} extended Deligne's approach to the class of hypersurfaces $X\subset P$, such that the sheaf of Du Bois differentials of $X$ and the sheaf of Barlet differentials of $X$ coincide. This happens only for few classes of singularities. The only known singular varieties for which this property holds are surfaces. Streenbrink \cite{SteAdj} gave three classes of examples, one of which are surfaces with ADE singularities \cite[Section 3]{SteAdj}.

To explain Steenbrink's results, let $X\subset P$ be a hypersurface, with at most isolated singularities. Let $\cT$ be the skyscraper sheaf supported on the singular locus, such that at each point $p$ the stalk $\cT_p$ is the Tjurina algebra of the singularity $(X,p)$.
 
The following proposition summarizes Steenbrink's method in the case where the ambient space $P$ is three-dimensional.
Note that if $X$ is a surface with at most ADE singularities then the mixed Hodge structure on $H^i(X)$ is pure of weight $i$. Hence it makes sense to define $H^{p,q}(X):=\Gr_F^p H^{p+q}(X)$.

\begin{proposition}\label{prpHpqCharSing} Let $P$ be a smooth compact three-dimensional K\"ahler manifold, and let $X\subset P$ be a surface with at most ADE singularities.
For all $G\subset \aut(P,X)$ we have $[H^{0,2}(X)]=[H^0(K_P(X))]$
and that $[H^{1,1}(X)]$ equals
\[\begin{array}{c} [H^{2,0}(P)]+[H^{2,2}(P)]+[H^{1,0}(X)]+[H^{1,2}(X)]-[H^{2,1}(P)]-[H^{2,3}(P)]\\-\chi_G(\Omega_P^2(X))+\chi_G(K_P(2X))-\chi_G(K_P(X))-\chi_G(\cT)\end{array}\]
in $K(\C[G])$.
\end{proposition}

\begin{proof}
Since ADE singularities are rational we get that
\[ H^{0,2}(X)=H^0(K_P(X)) \]
(see, e.g., \cite[Introduction]{SteAdj}).

The second equality follows from \cite{SteAdj}: 

Let $\Omega^2_X(\log X)$ be the kernel of $\Omega^2(X)\stackrel{d}{\to} K_P(2X)/K_P(X)$. Since $X$ has ADE singularities we have that the cokernel of $d$ is $\cT$ \cite[Section 2]{SteAdj}.
Define $\omega^1_X=\Omega^2_P(\log X)/\Omega^2_P$ to be the sheaf of Barlet $1$-forms on $X$.

Consider now the filtered de Rham complex $\tilde{\Omega}^\bullet_X$ on $X$, as introduced by Du Bois \cite{DuBois}.

Since $X$ has ADE singularities it follows from \cite[Section 4]{SteAdj} that $\Gr_F^1 \tilde{\Omega}^\bullet_X$ is concentrated in degree one, and in this degree it is isomorphic to $\tilde{\Omega}^1_X$. Moreover, in the same section Steenbrink shows that for a surface with ADE singularities we have $\tilde{\Omega}^1_X\cong \omega^1_X$.
 This implies  $H^i(\omega^1_X)=\Gr_F^1 H^{1+i}(X)$  and hence 
 \[ \chi_G(\omega^1_X)=[H^{1,0}(X)]-[H^{1,1}(X)]+[H^{1,2}(X)].\]
 The definition of $\omega^1_X$ yields the equality
 \[ \chi_G(\omega^1_X)=\chi_G(\Omega^2_P(\log X))-\chi_G(\Omega^2_{P}).\]
 Since $P$ is a smooth threefold we find that 
 \[ \chi_G(\Omega^2_{P})=[H^{2,0}(P)]-[H^{2,1}(P)]+[H^{2,2}(P)]-[H^{2,3}(P)].\]
 Using the definition of $\Omega^2_P(\log X)$ we find
 \[ \chi_G(\Omega^2_P(\log X))=\chi_G(\Omega^2_P(X))-\chi_G(K_{P}(2X))+\chi_G(K_P(X))+\chi_G(\cT).\]
 \end{proof}

 \begin{remark}\label{rmkWei} If $H^i(X)\cong H^i(P)$ holds for $i=1$ and $i=3$ then 
  \[ [H^{1,0}(X)]+[H^{1,2}(X)]=[H^{2,1}(P)]+[H^{2,3}(P)]\]
  If, moreover, $H^{2,0}(P)=0$ we have further simplifications in the formula from Proposition~\ref{prpHpqCharSing}.
  
  In case $P=\Ps(\cO\oplus f^*\cL^{-2}\oplus f^*\cL^{-3})$ and $X$ a Weierstrass model all these cancellations happen, and, moreover, $[H^{2,2}(P)]=2[\C]$ in $K(\C[G])$.
 \end{remark}
 
 \begin{corollary} \label{mainCor}Let $\cE$ be a direct sum of at least three line bundles on a smooth projective curve $C$. Let $X\subset \Ps(\cE)$ be a hypersurface. Let $f:C'\to C$ be a Galois cover. Let $X'=X\times_C C'\subset \Ps(f^*\cE)$ be the base-changed hypersurface. Assume that the natural map $H^i(\Ps(f^*(\cE)))\to H^i(X')$ is an isomorphism for $0\leq i< \dim X'$ and for $i=\dim X'$ this map is injective.
 
 If $X'$ is smooth then for each $p,q\in \Z$ there exist integers $a,b,c$, depending on $p$, $q$, the degrees of the direct summands of $\cE$ and the fiber degree of $X$, such that $[H^{p,q}(X')]=a[\C]+b\chi_G(\cO)+c[\C[G]]$.
 
 If $X'$ is surface with at most ADE singularities  for each $p,q\in \Z$ there exist integers $a,b,c$, depending on $p$, $q$, the degrees of the direct summand of $\cE$ and the fiber degree of $X$, such that $[H^{p,q}(X')]=a\C+b\chi_G(\cO)+c[\C[G]]+\delta [H^0(\cT)]$, where $\delta=0$ for $(p,q)\neq (1,1)$ and $\delta=1$ for $(p,q)=(1,1)$. 
 \end{corollary}

\begin{corollary} \label{WeiCor} Let $\cE$ be a direct sum of three line bundles. Let $W\subset \Ps(\cE)$ be a surface. Let $C'\to C$ be a Galois base change such that $W':=W\times_C C'$ is a surface with at most ADE singularities and such that $H^1(W')\cong H^1(\Ps)$. Let $X'$ be the desingularization of $W'$. Then $[H^{1,1}(W')]$ equals
\[ 2[\C]-\chi_G(\Omega^2(W'))+\chi_G(K_{\Ps(f^*\cE)}(2W'))-\chi_G(K_{\Ps(f^*\cE)}(W'))-\chi_G(\cT)\]
and
\[ [H^{1,1}(X')]=2[\C]-\chi_G(\Omega^2(W'))+\chi_G(K_{\Ps(f^*\cE)}(2W'))-\chi_G(K_{\Ps(f^*\cE)}(W'))\]
\end{corollary}
\begin{proof}
 The formula for $[H^{1,1}(W')]$ follows directly from Proposition~\ref{prpHpqCharSing}. The quotient $H^{1.1}(X')/H^{1,1}(W')$ is generated by the irreducible components of the resolution $X'\to W'$ and one easily checks that the representation induced by $G$-action on these irreducible components equlas $\cT$.
\end{proof}

\begin{remark}
Note that $[H^{1,1}(X')]$  depends only on the linear equivalence class of $W'$, and not on the singularities of $W'$. If $|W|$ is base point free then there is a different approach to obtain this statement. In this case $W'$ is the limit of a family of smooth surfaces, all of which are pulled back from $\Ps(\cE)$, and $W'$ has at most ADE singularities. In particular there is a simultaneous resolution of singularities of this family. The central fiber of this resolution is $X'$, and this implies the $\C[G]$-structure of $H^{p,q}(X')$ is the same as the one on the general member of this family.
\end{remark}

\section{The $\C[G]$-structure of the cohomology of Weierstrass models}\label{sectWeiModCG}
We want to apply the results of the previous section to the special case of Weierstrass models. In the first part of the section we only assume that $\cE$ is a direct sum of three line bundles. Let $C, C', X, X', \Ps_0,\Ps,\varphi,\varphi_0$ be as in the previous section. Assume that $\dim X=2$.

We want to determine the $\C[G]$-structure of $H^{1,1}(X)$ and of $H^{2,0}(X)$. By Corollary~\ref{WeiCor}  it suffices to determine the $\C[G]$-structure of  
\[ \chi_G(\Omega^2_{\Ps}(X)),\chi_G(K_{\Ps}(X)) \mbox{ and } \chi_G(K_{\Ps}(2X))\]
and the $\C[G]$-structure on $H^0(\cT)$.

We will determine the structure on $H^0(\cT)$ below. A strategy to calculate the three equivariant Euler characteristics is given in the proof of Proposition~\ref{prpCG}. The main ingredients are
\begin{enumerate}
 \item $\Omega^3_{\Ps}\cong \varphi^*\det(f^*\cE\otimes K_{C'})(-3)$ (adjunction).
\item $\Omega^2_{\varphi}\cong \varphi^*(\det(f^*\cE))(-3)$.
\item $ 0 \to \Omega^1_{\varphi}\to \varphi^*f^*\cE(-1)\to \cO_{\Ps}\to 0$ (Euler sequence).
\item $ 0 \to \Omega^1_{\varphi}\otimes \varphi^*K_{C'} \to \Omega^2_{\Ps}\to \Omega^2_{\varphi} \to 0.$
\end{enumerate}
The points (2)-(4) easily yield
\begin{lemma} Let $X\subset \Ps(\cE)$ be a hypersurface in $|(\varphi^*f^*\cL)(d)|$, fixed under $G$. Then $\chi_G(\Omega^2(X))$ equals
 \[\chi_G( \varphi^*f^*(\cL \otimes \det\cE)(d-3)) + \chi_G(\varphi^*f^*(\cL\otimes \cE)(d-1))-\chi_G(\varphi^*(f^*\cL\otimes K_{C'})(d)).\]
\end{lemma}

It turns out that if $\cE$ is a direct sum of line bundles then we can express all of the above equivariant  Euler characteristics in terms of equivariant Euler characteristics of sheaves of the form $(\varphi^*f^*\cF)(k)$ and $\varphi^*(f^*\cF \otimes K_{C'})(k)$, where $\cF$ is a direct sum of line bundles on $C$. The following lemmas are helpful in calculating $\chi_G$ of such sheaves.
 \begin{lemma} Suppose $\cE=\cO_{C'}\oplus \cL\oplus \cM$, with $\deg(\cL),\deg(\cM)\leq 0$.
  Then $\varphi_* \cO_{\Ps(\cE)} (t)$ is the pullback under $f^*$ of a direct sum of $\binom{k+2}{2}$ line bundles, such that the sum of the degrees equals
 \[ \frac{1}{6}t(t+1)(t+2)(\deg(\cL)+\deg(\cM)).\]
 \end{lemma}
\begin{proof}
Since $\cE=\cO_C \oplus \cL\oplus \cM$  there are canonical sections $X,Y,Z$ in $H^0(\varphi^*\cL^{-1}(1)), H^0(\varphi^*\cM^{-1}(1))$ and $H^0(\cO_{\Ps}(1))$ (cf. Section~\ref{sectWeiMod}).
 Note that 
 \[\varphi_*\cO(t)=\oplus_{0\leq i+j\leq t} (f^*\cL^i\otimes f^*\cM^j ) X^iY^jZ^{t-i-j}.\]
 Hence the sum of the degrees  equals
 \[ \sum_{0\leq i+j\leq t} (\deg(\cL)i+\deg(\cM)j)=\frac{1}{6}t(t+1)(t+2)(\deg(\cL)+\deg(\cM)).\]
\end{proof}

 \begin{lemma}\label{lemChar}
 Suppose $\cE=\cO_{C'}\oplus f^*\cL\oplus f^*\cM$, with $\deg(\cL),\deg(\cM)\leq 0$.
  Let $\cN$ be a line bundle on $C$. Let $t\geq 0$ be an integer.
 Set 
 \[ d=\binom{t+2}{3}(\deg(\cL)+\deg(\cM)) +\binom{t+2}{2} \deg(\cN).\]
  Then 
 \[ \chi_G(\varphi^* f^*\cN)(t))= d\C[G]+\frac{t+2}{2} \chi_G(\cO_{C'})
 \]
 and
\[ \chi_G(\varphi^*(K_{C'}\otimes f^*\cN)(t))= d\C[G]-\frac{t+2}{2} \chi_G(\cO_{C'}).
 \]
 \end{lemma}
 \begin{proof}
Since $R^i\varphi_*\cO(t)=0$ for $i>0$ we find that \[H^k(X,(\varphi^* f^*\cN)(t))=H^k(X,\varphi_*((\varphi^* f^*\cN)(t))).\] Combining this with the projection formula yields
\[ \chi_G((\varphi^* f^*\cN)(t))=\chi_G((f^*\cN)\otimes \varphi_*\cO(t)).\]
Since $\varphi_*\cO(t)$ is a direct sum of line bundles pulled back from $C$, the same holds for $f^*\cN\otimes \varphi_*\cO(t)$. The sum of the degree of the line bundles  on $C$ equals $d$. It follows now from Lemma~\ref{lemPal} that
\[ \chi_G((f^*\cN)\otimes \varphi_*\cO(t))=d\C[G]+\frac{t+2}{2} \chi_G(\cO_{C'}).\]
The Euler characteristic $\chi_G(\varphi^*(K_{C'}\otimes f^*\cN)(t))$ can be calculated similarly, by using Serre duality on $C'$.
 \end{proof}

From here on we assume that $\cE=\cO\oplus f^*\cL^{-2} \oplus f^*\cL^{-3}$ and that $W\in |\varphi_0^*\cL^6(3)|$ and hence that $X=W'\in |\varphi^*f^*\cL^6(3)|$.

We will now repeatedly apply Lemma~\ref{lemChar} to determine all the relevant Euler characteristics:
\begin{lemma} In $K(\C[G])$ we have
\[ \chi_G(K_\Ps(W'))= \deg(\cL)[\C[G]]-\chi_G(\cO_{C'}) \]
and
\[ \chi_G(K_\Ps(2W'))= 20\deg(\cL)[\C[G]]-10\chi_G(\cO_{C'}) \]
\end{lemma}
\begin{proof}
Note that \[K_{\Ps}=\varphi^*(\det(\cE)\otimes K_{C'}(-3)=\varphi^*(f^*\cL^{-5}\otimes K_{C'})(-3).\] Hence $K_{\Ps}(W')=\varphi^*f^*(\cL\otimes K_{C'})$.
From Lemma~\ref{lemChar} it now follows that $\chi_G(K_{\Ps}(W')=\deg(\cL)[\C[G]]-\chi_G(\cO_{C'})$.

Similarly $K_{\Ps}(W')=\varphi^*f^*(\cL^7\otimes K_{C'})(3)$. From Lemma~\ref{lemChar} it follows now that
\[ \chi_G(K_\Ps(2W'))= 20\deg(\cL)[\C[G]]-10\chi_G(\cO_{C'}).\]
\end{proof}

\begin{lemma}In $K(\C[G])$ we have
\[ \chi_G(\Omega^2_{\varphi}(W'))=\deg(\cL)[\C[G]]+\chi_G(\cO_{C'}).\]
\end{lemma}
\begin{proof}
 Note that $\Omega^2_{\varphi}(W')=(\varphi^*f^*\cL^{-5})(-3)\otimes \cL^6(3)=\varphi^*f^*(\cL)$. Lemma~\ref{lemChar} now yields
 \[ \chi_G(\Omega^2_{\varphi}(W'))=\deg(\cL)[\C[G]]+\chi_G(\cO_{C'}).\]
\end{proof}

\begin{lemma}In $K(\C[G])$ we have
\[ \chi_G(\varphi^*(K_{C'}(W'))=10\deg(\cL)[\C[G]]-10\chi_G(\cO_{C'}) \] 
\end{lemma}
\begin{proof}
Using $\varphi^*(K_{C'})(W')=\varphi^* (K_{C'}\otimes f^*\cL^6)(3)$ we obtain from Lemma~\ref{lemChar}
\[\chi_G(\varphi^*(K_{C'}(W')))=10\deg(\cL)[\C[G]]-10\chi_G(\cO_{C'}).\]
\end{proof}

\begin{lemma} In $K(\C[G])$ we have
 \[ \chi_G(\varphi^*(\cE \otimes K_{C'})(W')(-1))=18\deg(\cL)[\C[G]]-18\chi_G(\cO_{C'})\]
\end{lemma}
\begin{proof}
 Note that $\varphi^*(\cE \otimes K_{C'})(W')(-1)=\varphi^*(\cE \otimes K_{C'}\otimes f^*\cL^6)(2)$.
 Hence 
 \[\varphi^*(\cE \otimes K_{C'}\otimes f^*\cL^6)(2)=\varphi^*((f^*\cL^6\oplus f^*\cL^4\oplus f^*\cL^3)\otimes K_{C'})(2)
 \]
 From Lemma~\ref{lemChar} it follows that its Euler characteristic equals
 \[ 18\deg(\cL)[\C[G]]-18\chi_G(\cO_{C'})\]
\end{proof}

\begin{lemma} In $K(\C[G])$ we have
\[ \chi_G(\Omega^2(W'))=9\deg(\cL)[\C[G]]-7\chi_G(\cO_{C'}).\] 
\end{lemma}
\begin{proof}
From 
\[ 0 \to \Omega^1_{\varphi} \otimes \varphi^*K_{C'}(W') \to \Omega^2(W')\to \Omega^2_{\varphi}(W')\to 0
\]
and
\[ 0 \to \Omega^1_{\varphi} \otimes \varphi^*K_{C'}(W')\to \cE\otimes \varphi^*K_C(W')(-1)\to  \varphi^*K_C(W')\to 0.\]
It follows that $\chi_G(\Omega^2(W'))$ equals
\begin{eqnarray*} &&\chi_G(\Omega^2_{\varphi}(W'))+\chi_G(\cE\otimes \varphi^*K_C(W')(-1))-\chi_G(\varphi^* K_C(W'))\\&=&9\deg(\cL)[\C[G]]-7\chi_G(\cO_{C'}).\end{eqnarray*}
\end{proof}

Collecting everything we find:
\begin{proposition} We have the following identities in $K(\C[G])$:
\[ [H^{2,0}(W')]=[H^{2,0}(X')]=\deg(\cL)\C[G]+[H^0(K_{C'})]-[\C]\]
\[ [H^{1,1}(W')]=10\deg(\cL)[\C[G]]+2[H^0(K_{C'})]-[H^0(\cT)]\]
and
\[[H^{1,1}(X')]=10\deg(\cL)[\C[G]]+2[H^0(K_{C'})]\]
\end{proposition}

\begin{remark} A different proof for the formula for $H^{2,0}(X')$ can be found in \cite[Theorem 2.5]{PalMW}.\end{remark}

The $\C[G]$ action on $H^0(\cT)$ is hard to describe in general. However, if we make some assumption on the ramification locus then it simplifies a lot:
\begin{lemma}\label{lemTjurina}
 Suppose the ramification locus of $W'\to W$ does not intersect $W'_{\sing}$. Then
 \[ [H^0(\cT)]=\mu [\C[G]]\]
 where $\mu$ is the total Milnor number of $W$.
\end{lemma}

\begin{proof}
Let $\cT_{W}$ and $\cT_{W'}$ be the sheaves on $W$, resp. on $W'$, such that at each point $p$ the stalk is isomorphic to the Tjurina algebra at $p$.
The length of $\cT_{W}$ is the total Tjurina number of $W$, which equals the total Milnor number of $W$.

Since $\cT_{W'}$ is supported outside the ramification locus, we find that $\cT_{W'}$ is the pull back of $\cT_{W}$ and it consists of $\#G$ copies of $\cT_W$. In particular the $G$ action on $H^0(\cT_{W'})$ consists of $\mu$ copies of the regular representation.
\end{proof}

To obtain P\'al's upper bound for the Mordell-Weil rank we need the following
 \begin{proposition}[Shioda-Tate formula]\label{prpST}
 We have a short exact sequence of $\C[G]$-modules
 \[ 0 \to \C^2\oplus H^0(\cT) \to \NS(X') \to E(\C(C')) \to 0\]
\end{proposition}

\begin{proof}
Let $T\subset \NS(X')$ be the trivial sub-lattice, the lattice generated by the class of a fiber, the image of the  zero-section and the classes of irreducible components of reducible fibers. Shioda and Tate both showed that $E(\C(C'))$ is isomorphic to $\NS(X')/T$ as abelian groups.

The group $G$ acts on $T$, $\NS(X')$ and $E(\C')$, and from the construction of this map it follows directly that this isomorphism is $G$-equivariant.
Moreover the fiber components which do not intersect the zero-section are precisely the exceptional divisors of $X'\to W'$, i.e., they span as subspace isomorphic to $H^0(\cT)$. Since $G$ maps a fiber to a fiber, and fixes the zero section, we find 
 \[ 0 \to \C^2\oplus H^0(\cT) \to \NS(X') \to E(\C(C')) \to 0\]
is exact.
\end{proof}

\begin{theorem}\label{thmPal} Let $X\to C$ be an elliptic surface and let $f:C' \to C$ be a Galois cover such that the fibers of $\pi$ over the branch points of $f$ are smooth. Let $E$ be the general fiber of $\pi$. Let $\mu$ be the number of fiber-components not intersecting the zero-section, which equals the total Milnor number of $W$.

Then $E(\C(C'))\otimes_{\Z} \C$ is a quotient of a $\C[G]$-module $M$ such that
\[ [M]=(10\deg(\cL)-\mu) [\C[G]]+2[H^0(K_{C'})]-2[\C]\]
\end{theorem}

\begin{proof}
From Proposition~\ref{prpST} it follows $E(\C(C'))$ equals $\NS(X')/T(X')$.
 Now $\NS(X')$ (as $\C[G]$-module) is a quotient of $H^{1,1}(X')$. Hence $E(k(C'))$ is a quotient of $H^{1,1}(X')/T(X')$. 

Note that the Weierstrass model of $W'$ is the pullback of the Weierstrass model of $W$. In particular the minimal discriminant of $X'\to C'$ is the pullback of the minimal discriminant of $X\to C$. Our assumption on the singular fibers of $X\to C$ imply that the singular fibers are outside the ramification locus of $X'\to X$.
If $q\in W'_{\sing}$ then $q$ is a point on a singular fiber, hence $q$ is outside the ramification locus of $W'\to W$. Hence we may apply Lemma~\ref{lemTjurina} and obtain that  $[T(X')]=\mu [\C[G]]+2[\C]$.

From the previous section it follows that $[H^{1,1}(X')]=10\deg\cL [\C[G]]+2[H^0(K_{C'})]$, which yields the theorem.
\end{proof}

\begin{remark} If we allow the fibers over the branch points of $f$ to be semi-stable then the $\C[G]$-structure of $T$ is harder to describe. E.g., suppose we have a $I_1$ fiber over a branch point, with ramification index $2$ and $G=\Z/2\Z$. Then $X'\to C'$ has a $I_2$ fiber and this contributes a one dimensional vector space to $T$, on which $G$ acts via a non-trivial character.
\end{remark}

\begin{corollary}\label{corPal}Let $X\to C$ be an elliptic surface over a field $k$ of characteristic zero. Let $C' \to C$ be a Galois cover such that the fibers of $\pi$ over the branch points of $f$ are smooth. Let $E$ be the general fiber of $\pi$. Then
\[ \rank E(k(C')) \leq    \epsilon(G,k) (c_E+\frac{d_E}{6}+2g-2+s)\] 
\end{corollary}
\begin{proof}
As explained in \cite[Section 1]{PalMW} we may assume that $k=\C$ and that it suffices to prove that $E(\C(C'))$ is a quotient of $\C[G]^{c_E+\frac{d_E}{6}+2g-2+s}$.

From the Tate algorithm it follows that the number of fiber components in a singular fiber equals $v_p(\Delta)-1$ if the reduction is multiplicative and $v_p(\Delta)-2$ if the reduction is additive. Denote with $a$ the number of additive fibers and with $m$ the number of multiplicative fibers. Hence $\mu=d_E-m-2a$.
Now $c_E=m+2a$ and $d_E=12\deg(\cL)$. It follows from the previous theorem that $E(k(C'))$ is a quotient of the $\C[G]$-module $M$, with
\[[M]= (c_E+ \frac{d_E}{6})[\C[G]]+2[H^0(K_C')]-2[\C].\]

If $C'\to C$ is unramified that $H^0(K_C')=\C[G]^{g(C)}$. If $C'\to C$ is ramified then $H^0(\cO_Z)$ is a quotient of $\C[G]^s$, where $s$ is the number of critical values and we find $2H^0(K_C')$ is a quotient of $\C^{\oplus 2}\oplus \C[G]^{\oplus 2g-2+s}$

In both cases $E(\C(C'))$ is a quotient of $\C[G]^{\oplus c_E+\frac{d_E}{6}+2g-2+s}$.
\end{proof}

\bibliographystyle{plain}
\bibliography{remke2}

\begin{thebibliography}{10}

\bibitem{CW}
C.~{Chevalley} and A.~{Weil}.
\newblock {\"Uber das Verhalten der Integrale 1. Gattung bei Automorphismen des
  Funktionenk\"orpers.}
\newblock {\em {Abh. Math. Semin. Univ. Hamb.}}, 10:358--361, 1934.

\bibitem{DuBois}
Ph. Du~Bois.
\newblock Complexe de de {R}ham filtr\'e d'une vari\'et\'e singuli\`ere.
\newblock {\em Bull. Soc. Math. France}, 109:41--81, 1981.

\bibitem{EllMW}
J.~S. Ellenberg.
\newblock Selmer groups and {M}ordell-{W}eil groups of elliptic curves over
  towers of function fields.
\newblock {\em Compos. Math.}, 142:1215--1230, 2006.

\bibitem{Fab}
L.~A. Fastenberg.
\newblock Mordell-{W}eil groups in procyclic extensions of a function field.
\newblock {\em Duke Math. J.}, 89:217--224, 1997.

\bibitem{Fa}
L.~A. Fastenberg.
\newblock Computing {M}ordell-{W}eil ranks of cyclic covers of elliptic
  surfaces.
\newblock {\em Proc. Amer. Math. Soc.}, 129:1877--1883, 2001.

\bibitem{Fac}
L.~A. Fastenberg.
\newblock Cyclic covers of rational elliptic surfaces.
\newblock {\em Rocky Mountain J. Math.}, 39:1895--1903, 2009.

\bibitem{Har}
R.~Hartshorne.
\newblock {\em Algebraic geometry}, volume~52 of {\em Graduate Texts in
  Mathematics}.
\newblock Springer-Verlag, New York-Heidelberg, 1977.

\bibitem{Heijne}
B.~Heijne.
\newblock The maximal rank of elliptic {D}elsarte surfaces.
\newblock {\em Math. Comp.}, 81(278):1111--1130, 2012.

\bibitem{MiES}
R.~Miranda.
\newblock {\em The basic theory of elliptic surfaces}.
\newblock Dottorato di Ricerca in Matematica. ETS Editrice, Pisa, 1989.

\bibitem{Naka}
S.~Nakajima.
\newblock Galois module structure of cohomology groups for tamely ramified
  coverings of algebraic varieties.
\newblock {\em J. Number Theory}, 22:115--123, 1986.

\bibitem{PalManCst}
A.~P{\'a}l.
\newblock The {M}anin constant of elliptic curves over function fields.
\newblock {\em Algebra Number Theory}, 4:509--545, 2010.

\bibitem{PalMW}
A.~P{\'a}l.
\newblock Hodge theory and the {M}ordell-{W}eil rank of elliptic curves over
  extensions of function fields.
\newblock {\em J. Number Theory}, 137:166--178, 2014.

\bibitem{PSbook}
C.~A.~M. Peters and J.~H.~M. Steenbrink.
\newblock {\em Mixed {H}odge structures}, volume~52 of {\em Ergebnisse der
  Mathematik und ihrer Grenzgebiete. 3. Folge.}
\newblock Springer-Verlag, Berlin, 2008.

\bibitem{SilvTower}
J.~H. Silverman.
\newblock The rank of elliptic surfaces in unramified abelian towers.
\newblock {\em J. Reine Angew. Math.}, 577:153--169, 2004.

\bibitem{SteQua}
J.~H.~M. Steenbrink.
\newblock Intersection form for quasi-homogeneous singularities.
\newblock {\em Compositio Math.}, 34:211--223, 1977.

\bibitem{SteAdj}
J.~H.~M. Steenbrink.
\newblock Adjunction conditions for one-forms on surfaces in projective
  three-space.
\newblock In {\em Singularities and computer algebra}, volume 324 of {\em
  London Math. Soc. Lecture Note Ser.}, pages 301--314. Cambridge Univ. Press,
  Cambridge, 2006.

\bibitem{VerPoi}
J.~L. Verdier.
\newblock Le th\'eor\`eme de {L}e {P}otier.
\newblock In {\em Diff\'erents aspects de la positivit\'e ({S}\'em. {G}\'eom.
  {A}nal., \'{E}cole {N}orm. {S}up., {P}aris, 1972--1973)}, pages 68--78.
  Ast\'erisque, No. 17. Soc. Math. France, Paris, 1974.

\end{thebibliography}
\end{document}